\documentclass[10pt]{article}

\usepackage{amsmath}
\usepackage{amsthm}
\usepackage[all]{xy}
\usepackage{latexsym}

\newcommand{\FF}{\mathbf{F}}

\newcommand{\ZZ}{\mathbf{Z}}

\newcommand{\adj}{\operatorname{Adj}}

\newcommand{\cok}{\operatorname{coker}}

\newtheorem{theorem}{Theorem}
\newtheorem{proposition}{Proposition}
\theoremstyle{definition}
\newtheorem{definition}{Definition}

\newtheorem{example}{Example}

\begin{document}

\title{The adjoint group of an Alexander quandle.}
\author{F.J.-B.J.~Clauwens}

\maketitle
To an abelian group $M$ equipped with an automorphism $T$
one can associate a quandle $A(M,T)$ called its Alexander quandle.
It is given by the set $M$ together with the quandle operation $*$
defined by $y*x=Ty+x-Tx$.
To any quandle $Q$ one can associate a group $\adj(Q)$
called the \emph{adjoint group} of $Q$.
It is defined as the abstract group with one generator $e_x$ 
for each $x\in Q$ and one relation $e_{y*x}=e_x^{-1}e_ye_x$ for each $x,y\in Q$.

It is the purpose of this note to show that the adjoint group of an Alexander
quandle $Q(M,T)$ has an elegant description in terms of $M$ and $T$,
at least if the quandle is connected, which is the case if $1-T$ is invertible.
From this description one gets a similar description of the \emph{fundamental group}
of $Q(M,T)$ based at $0\in M$.
This note can be viewed as an exercise  inspired by \cite{eisermann5},
to which we refer for motivation and definitions.

\bigbreak

The adjoint group $A=\adj(A(M,T))$ acts from the right on $M$
by the formula $p\cdot e_x=p*x$.
This defines a homomorphism $\rho$ from $A$
to the group $G$ of quandle automorphisms of $A(M,T))$.
Thus $p\cdot e_0^{-1}=T^{-1}p$ and $p\cdot e_0^{-1}e_x=p+x-Tx$.
From one sees that
\begin{equation*}
p\cdot e_0^{-1}e_xe_0^{-1}e_y=p+x-Tx+y-Ty=p\cdot e_0^{-1}e_{x+y}
\end{equation*}
Therefore
\begin{equation}
\label{addform}
e_0^{-1}e_{x+y}=\gamma(x,y)e_0^{-1}e_xe_0^{-1}e_y
\end{equation}
for some $\gamma(x,y)\in \adj(Q)$ which acts trivially on $M$
and thus is an element of $K=\ker(\rho)$.
The group $K$ is a central subgroup of $A$ as explained in \cite{eisermann5}.

\bigbreak

From the definition of $\gamma(x,y)$ we see that $\gamma(0,y)=1$ and $\gamma(x,0)=1$ for all $x$ and $y$.
Furthermore the formulas
\begin{equation*}
\begin{split}
e_0^{-1}e_{x+y+z}
&=\gamma(x,y+z)e_0^{-1}e_xe_0^{-1}e_{y+z}\\
&=\gamma(x,y+z)e_0^{-1}e_x\gamma(y,z)e_0^{-1}e_ye_0^{-1}e_z\\
e_0^{-1}e_{x+y+z}
&=\gamma(x+y,z)e_0^{-1}e_{x+y}e_0^{-1}e_z\\
&=\gamma(x+y,z)e_0^{-1}\gamma(x,y)e_xe_0^{-1}e_ye_0^{-1}e_z
\end{split}
\end{equation*}
show that 
\begin{equation}
\gamma(x,y+z)\gamma(y,z)=\gamma(x+y,z)\gamma(x,y)
\text{ for all }x,y,z\in M
\end{equation}

\bigbreak

This shows that $\gamma$ is  a group $2$-cocycle for the group $M$
with values in $K$.
We will not use this: our purpose is not to show that $\gamma$
is a coboundary, 
but to show that it vanishes to a certain degree,
by exploiting its relation with $T$.
However if $\gamma$ were a coboundary
then in particular $\gamma(x,y)$ would be symmetric in $x$ and $y$.
This is one of the motivations to consider the map $\lambda\colon M\times M\to K$ defined by
\begin{equation}
\lambda(x,y)=\gamma(y,x)^{-1}\gamma(x,y)
=[e_0^{-1}e_y,e_0^{-1}e_x]
\end{equation}
\bigbreak
The defining relation for $A$ shows that 
$e_0e_xe_0^{-1}=e_{T^{-1}x}$
or equivalently
 $e_xe_0^{-1}=e_0^{-1}e_{T^{-1}x}$
for $x\in M$.
So we can rewrite
$e_{x+y}=\gamma(x,y)e_xe_0^{-1}e_y$
as
$e_{x+y}=\gamma(x,y)e_0^{-1}e_{T^{-1}x}e_y$.
In other words
\begin{equation}
e_ue_v=\gamma(Tu,v)^{-1}e_0e_{Tu+v}
\text{ for all }u,v\in M
\end{equation}
If we substitute this twice in the defining relation
we find that
\begin{equation*}
\begin{split}
&\gamma(Tu,v)^{-1}e_0e_{Tu+v}
=e_ue_v
=e_ve_{Tu+v-Tv}\\
&=\gamma(Tv,Tu+v-Tv)^{-1}e_0e_{Tv+Tu+v-Tv}\\
\end{split}
\end{equation*}
This implies that
$\gamma(Tu,v)=\gamma(Tv,Tu+v-Tv)$ for $u,v\in M$, 
in other words
\begin{equation}
\label{sym}
\gamma(x,y)=\gamma(Ty,x+y-Ty)\text{ for }x,y\in M
\end{equation}
and in particular
\begin{equation}
\label{gam}
\gamma(Ty,y-Ty)=1\text{ for }y\in M
\end{equation}
\bigbreak
We now switch to additive notation for $K$.
From (\ref{sym}) and the cocycle relation we find
\begin{equation*}
\begin{split}
&\gamma(u,v)+\gamma(v-Tv,u)\\
&=\gamma(Tv,v-Tv+u)+\gamma(v-Tv,u)\\
&=\gamma(Tv+v-Tv,u)+\gamma(Tv,v-Tv)
\end{split}
\end{equation*}
and in particular
\begin{equation}
\label{lamgam}
\lambda(u,v)=\gamma(u,v)-\gamma(v,u)=-\gamma(v-Tv,u)
\end{equation}
Thus if $\gamma$ were symmetric then $\lambda$ would vanish,
and so would $\gamma$ since $1-T$ is assumed to be invertible.
\bigbreak
Now we  look at the consequences for $\lambda$ of the cocycle condition for $\gamma$.
If we substitute (\ref{lamgam}) in the cocycle condition for $\gamma$ we find
\begin{equation*}
\lambda((1-T)^{-1}(x+y),z)+\lambda((1-T)^{-1}x,y)=\lambda((1-T)^{-1}x,y+z)+\lambda((1-T)^{-1}y,z)
\end{equation*}
and putting $x=u-Tu$, $y=v-Tv$ this yields
\begin{equation}
\label{cocgam}
\lambda(u+v,z)+\lambda(u,v-Tv)=\lambda(u,v-Tv+z)+\lambda(v,z)
\end{equation}
On the other hand subtracting two instances of the cocycle condition for $\gamma$
\begin{equation*}
\begin{split}
&\gamma(u,v+z)+\gamma(v,z)=\gamma(u+v,z)+\gamma(u,v)\\
&\gamma(z,v+u)+\gamma(v,u)=\gamma(z+v,u)+\gamma(z,v)\\
\end{split}
\end{equation*}
we find
\begin{equation}
\label{coclam}
\lambda(u+v,z)+ \lambda(u,v)=\lambda(u,v+z)+ \lambda(v,z)
\end{equation}
Substracting (\ref{coclam}) from (\ref{cocgam}) we find
\begin{equation}
\label{indep}
\lambda(u,v-Tv)-\lambda(u,v)=\lambda(u,v-Tv+z)-\lambda(u,v+z)
\end{equation}
This means that the right hand side of (\ref{indep}) does not depend on $z$;
in particular it has the same value for $z=-v$.
Thus using the fact that $\lambda(u,0)=0$ we can rewrite (\ref{indep}) as
\begin{equation}
\label{indep2}
\lambda(u,-Tv)=\lambda(u,v-Tv+z)-\lambda(u,v+z)
\end{equation}
Substituting $a=v+z$ and $b=-Tv$ this yields
\begin{equation}
\label{indep3}
\lambda(u,b)=\lambda(u,a+b)-\lambda(u,a)
\end{equation}
\bigbreak
We have just proved that $\lambda$ is additive in its second coordinate.
Since $\lambda$ is skew-symmetric it is in fact bi-additive.
Thus (\ref{lamgam}) and the invertibility of $1-T$ imply that $\gamma$ is bi-additive.
Moreover using (\ref{gam}) we can simplify (\ref{sym}) to
\begin{equation}
\gamma(x,y)=\gamma(Ty,x)\text{ for  all }x,y
\end{equation}
This motivates the following definition and theorem.
\begin{definition}
Define $\tau\colon M\otimes M\to M\otimes M$ by the formula $\tau(x\otimes y)=Ty\otimes x$.
Define $S(M,T)$ as $\cok(1-\tau)$. 
Thus $\gamma$ can be viewed as a map from $S(T,M)$ to $K$.
Finally define $F(M,T)$ as the set $\ZZ\times M\times S(M,T)$
with the multiplication given by
\begin{equation*}
(k,x,\alpha)(m,y,\beta)=(k+m,T^mx+y,\alpha+\beta+[T^mx\otimes y])
\end{equation*}
\end{definition}
\begin{theorem}
The groups $\adj(A(M,T))$ and $F(M,T)$ are isomorphic.
\end{theorem}
\begin{proof}
We define $\phi\colon\adj(A(M,T))\to F(M,T)$ by setting
$\phi(e_x)=(1,x,0)$.
To see that this is well defined we have to check the following:
\begin{equation*}
\begin{split}
&\phi(e_x)\phi(e_{y*x})
=(1,x,0)(1,Ty+x-Tx,0)\\
&=(2,Tx+(Ty+x-Tx),[Tx\otimes(Ty+x-Tx)])\\
&=(2,Ty+x,[Ty\otimes x])
=(1,y,0)(1,x,0)
=\phi(e_y)\phi(e_x)\\
\end{split}
\end{equation*}
which is the case since $[Tx\otimes Ty]=[Ty\otimes x]$ 
and $[Tx\otimes Tx]=[Tx\otimes x]$.\\
We define $\psi\colon F(M,T)\to \adj(A(M,T))$ 
by setting
$\psi(k,x,\alpha)=e_0^{k-1}e_x\gamma(\alpha)^{-1}$.
To see that $\psi$ is a homomorphism we have to check the following:
\begin{equation*}
\begin{split}
&\psi(k,x,\alpha)\psi(m,y,\beta)
=e_0^{k-1}e_x\gamma(\alpha)^{-1}
e_0^{m-1}e_y\gamma(\beta)^{-1}\\
&=e_0^{k-1}e_x
e_0^m e_0^{-1}e_y
\gamma(\alpha)^{-1}\gamma(\beta)^{-1}
=e_0^{k-1}e_0^m e_{T^mx}
e_0^{-1}e_y
\gamma(\alpha+\beta)^{-1}\\
&=e_0^{k+m-1}e_{T^mx+y}\gamma(T^mx\otimes y)^{-1}\gamma(\alpha+\beta)^{-1}\\
&=\psi(k+m,T^mx+y,\alpha+\beta+[T^mx\otimes y])\\
\end{split}
\end{equation*}
which is the case  $e_ze_0^{-1}e_y=e_{z+y}\gamma[z\otimes y]^{-1}$ for $z=T^mx$ by (\ref{addform}).\\
From $\psi(\phi(e_x))=\psi(1,x,0)=e_x$  we see that $\psi\phi=1$.
The other composition requires more work;
first we compute
\begin{equation*}
\begin{split}
&\phi(\gamma[u\otimes v]^{-1})
=\phi(e_{u+v}^{-1}e_ue_0^{-1}e_v)
=(1,u+v,0)^{-1}(1,u,0)(1,0,0)^{-1}(1,v,0)\\
&=(-1,-T^{-1}(u+v),[(u+v)\otimes(u+v)])(1,u,0)(-1,0,0)(1,v,0)\\
&=(-1,-T^{-1}(u+v),[(u+v)\otimes(u+v)])(1,u+v,[u\otimes v])
=(0,0,[u\otimes v])
\end{split}
\end{equation*}
which shows that $\phi(\gamma(\alpha)^{-1})=(0,0,\alpha)$ for all $\alpha$.
From this we get
\begin{equation*}
\phi(\psi(k,x,\alpha))
=\phi(e_0^{k-1})\phi(e_x)\phi(\gamma(\alpha)^{-1})
=(k-1,0,0)(1,x,0)(0,0,\alpha)
=(k,x,\alpha)
\end{equation*}
so we find that $\phi\psi=1$.
\end{proof}
\bigbreak
For any quandle $Q$ there is a unique homomorphism $\epsilon\colon\adj(Q)\to\ZZ$
such that $\epsilon(e_x)=1$ for all $x\in Q$;
the kernel is denoted by $\adj(Q)^o$.
It is clear that $\epsilon(\alpha)=0$ for all $\alpha$,
so $\epsilon(\psi(k,x,\alpha))=k$. 
Therefore under $\psi$ the subgroup $\adj(A(M,T))^o$ of $\adj(A(M,T))$
corresponds to the subgroup $F(M,T)^o$ of $F(M,T)$
consisting of the triples $(0,x,\alpha)$.
Note that on $F(M,T)^o$ the multiplication simplifies to
\begin{equation*}
(0,x,\alpha)(0,y,\beta)=(0,x+y,\alpha+\beta+[x\otimes y])
\end{equation*}
For any quandle the fundamental group based at $q\in Q$
is defined as $\pi_1(Q,q)=\{g\in\adj(Q)^o\;|\;q\cdot g=q\}$.
For these definitions we refer to \cite{eisermann5}.
In order to describe this in terms of $(M,T)$ for the case
$Q=A(M,T)$ we need to translate the action of $\adj(A(M,T))$ on $M$
into an action of $F(M,T)$ on $M$.

One can easily check that 
$0\cdot\psi(k,x,\alpha)=x-Tx$
for all $k$, $x$ and $\alpha$.
This implies that $0\cdot(0,x,\alpha)=0$ if and only if $x=0$,
which means that
$\pi_1(A(M,T),0)$ is isomorphic to $S(M,T)$.

\bigbreak
\begin{example}
Let $\FF$ be a field,
let $M=\FF[t]/(t^2+at+b)$
and let $T$ be multilpication by the class of $t$.
Then $T$ is an automorphism if $b\not=0$
and $A(M,T)$ is connected if $1+a+b\not=0$.
In this case $S(M,T)$ isomorphic to $K/(b^2+ab-a-1)$.
Thus $A(M,T)$ is simply connected if $b^2+ab-a-1\not=0$.
The entry for $\FF=\ZZ/(3)$ and $f(t)=t^2-t+1$ 
in the table on page 49 of \cite{carterjkls2} is not compatile with this,
but it is a misprint.
\end{example}

\end{document}